\documentclass[11pt, leqno]{amsart}

\usepackage{amsmath}
\usepackage{amssymb,amsfonts}
\usepackage{cancel}
\usepackage{comment}
\usepackage{enumerate}
\usepackage{enumitem}
\usepackage{geometry}
\usepackage{hyperref}
\usepackage{a4wide}
\usepackage{mathrsfs}
\usepackage[dvipsnames]{xcolor}
\usepackage[all,arc]{xy}

\newtheorem{thm}{Theorem}[section]

\newtheorem{prop}[thm]{Proposition}
\newtheorem{lem}[thm]{Lemma}

\theoremstyle{definition}

\theoremstyle{remark}
\newtheorem{rem}[thm]{Remark}

\makeatletter
\let\c@equation\c@thm
\makeatother
\numberwithin{equation}{section}

\newcommand{\sumao}{\sum_{\alpha=n+1}^{n+p}}
\newcommand{\sumbo}{\sum_{\beta=n+1}^{n+p}}
\newcommand{\sumat}{\!\!\sum_{\alpha=n+2}^{n+p}\!\!}
\newcommand{\sumbt}{\!\!\sum_{\beta=n+2}^{n+p}\!\!}
\newcommand{\sumabo}{\!\!\sum_{\alpha,\beta=n+1}^{n+p}\!\!}
\newcommand{\sumabt}{\!\!\sum_{\alpha,\beta=n+2}^{n+p}\!\!}
\newcommand{\tr}{\text{tr}}
\newcommand{\dvol}{d\text{vol}}
\newcommand{\Vol}{\text{Vol}}
\newcommand{\Id}{\text{Id}}
\newcommand{\diverg}{\text{div}}

\title{Classification of quadratically pinched self-shrinkers in higher codimension}

\thanks{The authors are partially supported by INdAM-GNSAGA}

\subjclass[2020]{53C42, 53C21}

\keywords{Self-shrinkers, high codimension, quadratic pinching, weighted manifolds, }

\author[Debora Impera]{Debora Impera}
\address[Debora Impera]{Dipartimento di Scienze Matematiche ``Giuseppe Luigi Lagrange", Politecnico di Torino, Corso Duca degli Abruzzi, 24, Torino, Italy, I-10129}
\email{debora.impera@polito.it}

\author[Michele Rimoldi]{Michele Rimoldi}
\address[Michele Rimoldi]{Dipartimento di Scienze Matematiche ``Giuseppe Luigi Lagrange", Politecnico di Torino, Corso Duca degli Abruzzi, 24, Torino, Italy, I-10129}
\email{michele.rimoldi@polito.it}

\author[Francesco Ruatta]{Francesco Ruatta}
\address[Francesco Ruatta]{Dipartimento di Scienze Matematiche ``Giuseppe Luigi Lagrange", Politecnico di Torino, Corso Duca degli Abruzzi, 24, Torino, Italy, I-10129}
\email{francesco.ruatta@polito.it}

\hypersetup{
    colorlinks=true,
    linkcolor=RawSienna,
    citecolor=RawSienna,
    filecolor=magenta,      
    urlcolor=cyan,
    pdftitle={Overleaf Example},
    pdfpagemode=FullScreen,
}

\allowdisplaybreaks

\begin{document}

\begin{abstract}\label{abstract}
We classify properly immersed self-shrinkers of the mean curvature flow in arbitrary codimension under a quadratic pinching condition of Andrews-Baker type on the second fundamental form that is preserved along the flow. Under this assumption, such self-shrinkers reduce effectively to codimension one and are therefore generalized self-shrinking cylinders. In contrast to previous works, our approach is purely elliptic: it relies on parabolicity in a weighted setting and is tailored specifically to self-shrinkers, rather than to general ancient solutions of the flow. This allows us to avoid assuming any uniform pinching condition, to treat in any dimension the sharp Andrews-Baker pinching constant $\frac{4}{3n}$ and hence to sharpen, in the self-shrinker setting, the pinching constants appearing in recent classification results for ancient solutions.
\end{abstract}

\maketitle


\section{Introduction}\label{section_introduction}

Let $p\geq1$ and let $x_{0}:M^n\to\left(\mathbb{R}^{n+p},\langle\cdot,\cdot\rangle\right)$ be a smooth isometric immersion of a connected $n$-dimensional Riemannian manifold without boundary as a complete submanifold in the $(n+p)$-dimensional Euclidean space. A one-parameter family $\{x(\cdot,t)\}_{t\in[t_0,T)}$ of smooth immersions with corresponding images $M_t\doteq x(\cdot,t)(M)$ is said to be a solution of the \textit{mean curvature flow} (MCF) starting from $M_{t_0}$ if it satisfies the initial value problem
\begin{equation}\label{MCF}
    \begin{cases}
        \partial_t x(\cdot,t)=H(\cdot,t),
        \\x(\cdot,t_0)=x_0.
    \end{cases}
\end{equation}
Here $H(\cdot,t)$ denotes the mean curvature vector of $M_t$ at $x(\cdot, t)$, defined as the trace of the second fundamental form $B$. 

A solution to \eqref{MCF} is called ancient if $t_0=-\infty$. The simplest examples of ancient solutions are given by homothetically shrinking solutions. A {\it self-shrinker} of the MCF is defined as the time-slice at $t=-1$ of a self-shrinking solution to \eqref{MCF} that becomes extinct at the space-time point $(0,0)$. This notion admits an alternative characterization: a smooth isometric immersion $x:M^n\to(\mathbb{R}^{n+p},\langle\cdot,\cdot\rangle)$ is a self-shrinker if it satisfies the quasilinear elliptic system
\begin{equation}\label{self_shrinker_equation}
    H=-\frac{1}{2} x^\perp,
\end{equation}
where $(\cdot)^{\perp}$ denotes the projection on the normal bundle of $M$.

It is well-known that self-shrinkers play a fundamental role in the study of the mean curvature flow, as they model singularity formation and arise naturally as tangent flows at singular points. 

In the codimension one setting, extensive work has been devoted to classification results under convexity-type assumptions. In particular, immersed self-shrinkers of dimension one in $\mathbb R^2$ are either straight lines through the origin, the circle $\mathbb{S}^{1}(\sqrt{2})$, or locally mean convex curves known as {\it Abresch-Langer curves} \cite{abresch_langer_1986}, named after their discoverers U. Abresch and J. Langer.
This classification was later extended to higher dimensions. For $n\geq 2$, G. Huisken, \cite{huisken_1990_asymptotic_behavior}, proved that the only smooth, closed self-shrinkers $M^n\rightarrow\mathbb R^{n+1}$ ($n\geq 2$) with nonnegative mean curvature are the round spheres $\mathbb{S}^n(\sqrt{2n})$. Moreover, in the complete (possibly non-compact) setting, Huisken showed in \cite{huisken_1993_local_global_behaviour} that smooth self-shrinkers $M^n\rightarrow\mathbb R^{n+1}$ ($n\geq 2$) with nonnegative mean curvature, polynomial volume growth, and bounded second fundamental form must be isometric either to a product $\Gamma\times\mathbb R^{n-1}$, where $\Gamma$ is an Abresch-Langer curve, or to a generalized cylinder $\mathbb S^k(\sqrt{2k})\times\mathbb R^{n-k}$ ($0\leq k\leq n$). Subsequently, T. Colding and W. Minicozzi, \cite{colding_minicozzi_2012_generic_singularities}, showed that the boundedness hypothesis on $|B|$ can be removed, leading to the following:

\begin{thm}[\cite{colding_minicozzi_2012_generic_singularities}, Theorem 0.17]\label{colding_minicozzi_theorem}
    Let $x:M^n\to\mathbb{R}^{n+1}$ be a complete immersed self-shrinker without boundary. If $M$ has polynomial volume growth and nonnegative mean curvature ($H\geq 0$), then it is isometric either to the product $\Gamma\times \mathbb R^{n-1}$, where $\Gamma$ is an Abresch-Langer curve, or to $\mathbb S^k(\sqrt{2k})\times\mathbb R^{n-k}$ ($0\leq k\leq n$).
\end{thm}

In {\it higher codimension} $p>1$, the situation is much more subtle because of the presence of the normal curvature and of the fact that the second fundamental form is a normal bundle valued tensor and hence no obvious notion of convexity is available. In particular the sole assumption $|H|>0$ is not strong enough to guarantee a classification result. Indeed observe, for instance, that minimal submanifolds of any self-shrinking round sphere $\mathbb{S}^{n+k}(\sqrt{2(n+k)})$ with $k\leq p-1$ are self-shrinkers in $\mathbb{R}^{n+p}$ satisfying $|H|>0$. Hence, in order to obtain similar classification results, some extra assumption is needed (e.g. flat normal bundle, parallel principal normal, uniformly bounded geometry, low entropy), see for instance \cite{ArezzoSun, smoczyk_2005, coldingminicozzi_Par, Lee}. The first result of this paper is an improvement of a classification result in \cite{smoczyk_2005} substituting the uniformly bounded geometry condition appearing there with the properness of the immersion. More precisely, we prove the following: 

\begin{thm}\label{thm:extensionsmoczyck}
    Let $x:M^n\to \mathbb R^{n+p}$ ($n\geq 2$, $p\geq 1$) be a properly immersed self-shrinker without boundary. If $|H|> 0$ and $M$ has parallel principal normal $\nu\doteq H/|H|$, then it must belong to one of the following classes:
    \begin{equation*}
        M=\Gamma\times\mathbb R^{n-1},\qquad M=\widetilde M^r\times\mathbb R^{n-r},
    \end{equation*}
    where $\Gamma$ is an Abresch-Langer curve and $\widetilde M^r$ is a complete, minimal submanifold of the sphere $\mathbb{S}^{p+r-1}(\sqrt{2r})\subseteq \mathbb R^{p+r}$, where $0<r=\text{rank}(B_\nu)\leq n$ denotes the rank of the principal second fundamental form $B_{\nu}\doteq\langle B, \nu\rangle$.
\end{thm}

Even though in higher codimension there is no obvious notion of convexity, 
a natural substitute is provided by the so-called 
\emph{quadratic pinching condition}. 
We say that a submanifold is 
$c_n$-quadratically pinched if it satisfies the pointwise inequality
\[
|B|^2 \leq c_n |H|^2.
\]
In codimension $p=1$, suitable choices of the pinching constant $c_n$ are 
closely related to convexity-type assumptions. 
More precisely, it is well known that a hypersurface which is uniformly
$\frac{1}{n-k}$-quadratically pinched, for some $1 \leq k \leq n-1$, 
is in fact uniformly $k$-convex; see \cite[Lemma 5.1]{HuiskenSinestrari}.

Besides its relation to convexity, a further crucial aspect in the choice of the 
pinching constant in our setting is whether the condition is preserved along mean curvature flow. Our specific choice of the pinching constant is motivated by the work of 
B.~Andrews and C.~Baker, who proved in Theorem $2$ of \cite{andrews_baker_2010} 
that the $c_{n}$-quadratic pinching is preserved along the mean curvature flow 
provided that $c_n < \frac{4}{3n}$. Moreover, the quadratic pinching condition has proved to be a powerful structural assumption, leading to classification results for ancient solutions, convexity estimates in higher codimension, and even the development of mean curvature flow with surgery and topological classification results (see e.g. \cite{Naff}, \cite{LN1}, \cite{LN2}, \cite{LN3}).

Focusing on the borderline case 
\[
c_n= \frac{4}{3n},
\]
in the main result of this paper, we classify connected, properly immersed, and $c_{n}$-pinched self-shrinkers in higher codimension. Precisely, we will prove the following: 

\begin{thm}\label{main_theorem}
    Let $x:M^n\to \mathbb R^{n+p}$ ($n\geq 2$, $p\geq 1$) be a properly immersed self-shrinker without boundary. If $M$ satisfies a $\frac{4}{3n}$-quadratic pinching condition then either $H\equiv 0$ or $|H|>0$. In the first case $M$ is isometric to an $n$-dimensional plane. In the second case the principal normal $\nu\doteq H/|H|$ is parallel in the normal bundle, $M$ has effective codimension one and it is isometric to a generalized cylinder $S^k(\sqrt{2k})\times\mathbb R^{n-k}$, with $\left\lceil\frac{3n}{4}\right\rceil\leq k\leq n$. 
\end{thm}

The above result can be viewed as an improvement of Corollary 1.7 in \cite{leenaffzhu}. In \cite{leenaffzhu}, the classification is obtained  through a parabolic analysis of ancient solutions to the mean curvature flow and requires a \textit{uniform} quadratic pinching condition with dimension-dependent constant
\[
c_n=\begin{cases}
    \frac{3(n+1)}{2n(n+2)} & \textrm{if } n\leq 7,\\
    \frac{4}{3n} & \textrm{if } n\geq 8.
\end{cases}\]
Our approach is conceptually different. Rather than relying on the parabolic analysis of ancient solutions, we work directly at the level of self-shrinkers, from an elliptic point of view. This allows us to treat the borderline Andrews-Baker constant $\frac{4}{3n}$ in every dimension, while also removing the assumption of uniform pinching. In particular, the codimension reduction phenomenon appears here as an intrinsic rigidity property of self-shrinkers at the critical quadratic pinching constant.

\begin{rem}
As it will be clear from the proofs, the assumption on the properness of the immersion, via \cite{DX}, \cite{CZ}, is used repeatedly to ensure the convergence of various weighted integrals, thereby justifying integration by parts. As a matter of fact, Theorem \ref{main_theorem} holds substituting the properness assumption with completeness and a weighted $L^2$-condition on the norm of the second fundamental form (see e.g. \cite{Rimoldi_Class}). 
\end{rem}

The paper is organized as follows. In Section \ref{section_equations_and_setting}, we recall basic definitions and equations for submanifolds, fix the notation, and recall some basic facts about weighted manifolds that will be used throughout the paper. Section \ref{section_simons_equations} is devoted to the derivation of Simons-type equations for the squared norms of the mean curvature vector field and the second fundamental form. The final two sections contain, respectively, the proofs of Theorem \ref{thm:extensionsmoczyck} and Theorem \ref{main_theorem}.

\section{Geometric setup and notation}\label{section_equations_and_setting}

In this section, we recall some basic definitions and equations for submanifolds and establish the notation used through the paper.

Let $x:M^n\to (\mathbb R^{n+p},\langle\cdot,\cdot\rangle)$ be a smooth immersion of an $n$-dimensional manifold $M^n$ into the $(n+p)$-dimensional Euclidean space. The immersion $x$ induces a metric $g$ on $M$ via the pullback. We denote by $\overline\nabla$ and $\nabla$ the Levi-Civita connections on $\mathbb R^{n+p}$ and $M$, respectively; by $TM$ and $NM$ the tangent and normal bundles of $M$ in $\mathbb R^{n+p}$; and by $\Gamma(TM)$ and $\Gamma(NM)$ their spaces of smooth sections. 
The $(1,3)$-curvature tensor $R$ of $(M, \nabla)$ is defined with the sign convention 
\begin{equation*}
    R(X,Y)Z\doteq \nabla_X(\nabla_YZ)-\nabla_Y(\nabla_XZ)-\nabla_{[X,Y]}Z.
\end{equation*}
The associated $(0,4)$-curvature tensor is given by $R(X,Y,Z,W)\doteq \langle R(X,Y)W,Z\rangle$. Unless otherwise stated, $\Delta$ and $\diverg$ denote the Laplacian and the divergence on the submanifold $M$, respectively. 

We will make use of the following conventions on the range of indices:
\begin{equation*}
    1\leq i,j,k,\ldots\leq n,\qquad n+1\leq\alpha,\beta,\gamma,\ldots\leq n+p.
\end{equation*}
Let $\{e_i\}_{i=1}^n$ and $\{\nu_\alpha\}_{\alpha=n+1}^{n+p}$ represent local orthonormal frames for the tangent and normal bundles, respectively. If $H$ is nowhere vanishing, we set $\nu_{n+1}$ to be the principal normal $H/|H|$, denoted simply as $\nu$ when there is no risk of confusion. Unless otherwise specified, sums over Latin indices range from $1$ to $n$, while sums over Greek indices range from $n+1$ to $n+p$.

\subsection{The second fundamental form}
The {\it second fundamental form} $B$  of $M$ corresponds to the normal part of the ambient covariant derivative:
\begin{equation*}
    B:\Gamma(TM)\times \Gamma(TM)\longrightarrow \Gamma(NM),\qquad B(X,Y)\doteq (\overline\nabla_XY)^\perp=\overline\nabla_XY-\nabla_XY.
\end{equation*}
In terms of local coordinates on $M$, $B$ is expressed as 
\begin{equation*}
    B=\sum_\alpha h^\alpha\otimes\nu_\alpha,
\end{equation*}
where the $h^\alpha$'s are symmetric $2$-tensor with components 
\begin{equation*}
    h_{ij}^\alpha=\langle \overline\nabla_{e_i}e_j, \nu_\alpha \rangle.
\end{equation*}

The {\it mean curvature vector} $H$ of $M$ is the trace of the second fundamental form $B$. 
Under the assumption of nowhere vanishing mean curvature vector, the tensors $h^\alpha$ satisfy $\tr\: h^{n+1}=|H|$ and $\tr\: h^{\alpha}=0$ for $\alpha\geq n+2$. In this setting, the {\it tracefree second fundamental form}, defined by $\mathring B\doteq B-\frac{1}{n}g\otimes H$, can be written component-wise as $\mathring B=\sum_\alpha\mathring h^\alpha\otimes\nu_\alpha$. Here, the coefficients are $\mathring h^{n+1}=h^{n+1}-\frac{|H|}{n}\operatorname{Id}$ and $\mathring h^\alpha=h^\alpha$ for $\alpha\geq n+2$. 
Defining:
\begin{equation*}
    B_\nu\doteq h^{n+1}\otimes\nu,\qquad \mathring B_\nu\doteq \mathring h^{n+1}\otimes\nu,\qquad B_I\doteq \sumat h^\alpha\otimes\nu_\alpha,
\end{equation*}
it is clear that the following decompositions hold
\begin{equation*}
    |B_I|^2=\sumat |h^\alpha|^2=|B|^2-|B_\nu|^2,\qquad |\mathring B_\nu|^2=|B_\nu|^2-\frac{|H|^2}{n}.
\end{equation*}

Alternatively, we can describe the second fundamental form through the connection $1$-forms (see e.g. \cite{AliasMastroliaRigoli}). Let $\{\omega^A\}_{A=1}^{n+p}$ be the co-frame dual to $\{e_1,\ldots,e_n,\nu_{n+1},\ldots,\nu_{n+p}\}$ and $\{\omega^{A}_{B}\}_{A,B=1}^{n+p}$ the connection $1$-forms of $\mathbb R^{n+p}$. The latter are described by exterior differentiation of the $\omega^A$'s , and are uniquely defined by Cartan's first structure equations: $d\omega^A=\sum_B\omega^B\wedge \omega^{A}_{B}$ and $ \omega^{A}_{B}+\omega^{B}_{A}=0$. Restricting to $M$, there exist smooth functions $h_{ij}^\alpha$ such that $\omega^{\alpha}_{i}=\sum_jh_{ij}^\alpha\:\omega^j$, with $h_{ij}^\alpha=h_{ji}^\alpha$.
These are precisely the components of the second fundamental form:
\begin{equation*}
    B=\sum_{i,j,\alpha}h^\alpha_{ij}\:\omega^i\otimes\omega^j\otimes\nu_\alpha=\sum_\alpha h^\alpha\otimes\nu_\alpha.
\end{equation*} 
The first and second covariant derivatives of $h^\alpha_{ij}$ are defined, respectively, by:
\begin{align}
    \sum_k h^\alpha_{ijk}\:\omega^k&\doteq dh^\alpha_{ij}-\sum_k h^\alpha_{ik}\:\omega^{k}_{j}-\sum_k h^\alpha_{kj}\:\omega^{k}_{i}+\sum_\beta h^\beta_{ij}\:\omega^{\alpha}_\beta,\label{first_sff_covariant_derivative_connection_forms}
    \\\sum_l h^\alpha_{ijkl}\:\omega^l&\doteq dh^\alpha_{ijk}-\sum_l h^\alpha_{ijl}\:\omega^{l}_{k}-\sum_l h^\alpha_{ilk}\:\omega^{l}_{j}-\sum_l h^\alpha_{ljk}\:\omega^{l}_{i}+\sum_\beta h^\beta_{ijk}\:\omega^{\alpha}_{\beta}.\nonumber
\end{align}
The latter obey the following commutation rule:
\begin{equation}\label{second_covariant_derivatives_sff_commutation_rule}
    h^\alpha_{ijkl}=h^\alpha_{ijlk}+\sum_m h^\alpha_{mj}R_{mikl}+\sum_m h^\alpha_{im}R_{mjkl}+ \sum_\beta h^\beta_{ij}R_{\beta\alpha kl}.    
\end{equation}  

\subsection{Fundamental equations}
The intrinsic and extrinsic geometry of the submanifold are related by the following equations (see e.g. \cite{do_carmo_1992_riemannian_geometry}, chapter $6$): 
\begin{itemize}
    \item {\it Gauss' equation}: 
        for $X,Y,Z,W\in\Gamma(TM)$,
        \begin{equation*}\label{gauss_equation_general}
            R(X,Y,Z,W)=\langle B(X,Z),B(Y,W)\rangle-\langle B(X,W),B(Y,Z)\rangle.
        \end{equation*} 
        Expressing this equation in local coordinates, we obtain
        \begin{equation*}\label{gauss_equation}
            R_{ijkl}=\sum_\beta h^\beta_{ik}h^\beta_{jl}-\sum_\beta h^\beta_{il}h^\beta_{jk}.
        \end{equation*}
        
    \item {\it Weingarten's equation}:   
        for $X,Y\in\Gamma(TM)$ and $\xi\in\Gamma(NM)$,
        \begin{equation*}\label{weingarten_equation}
            \overline\nabla_X\xi=-S_\xi(X)+\nabla^\perp_X\xi,
        \end{equation*}
        where $\nabla^\perp$ is the connection on the normal bundle $NM$ and $S_\xi$ is the {\it shape operator} associated to the normal direction $\xi$, defined by
        \begin{equation*}\label{shape_operator_definition}
            S_\xi:\Gamma(TM)\longrightarrow \Gamma(TM),\qquad\langle S_\xi(X),Y\rangle=\langle B(X,Y),\xi\rangle.
        \end{equation*}
        \begin{rem}\label{trace_norm_SH_remark}
            Let us restrict our attention to the case $\xi=H$. A direct computation shows that the trace of $S_H$ yields the squared norm of the mean curvature vector:
            \begin{equation*}
                \tr\:S_H=\sum_i\langle S_H(e_i),e_i\rangle=\sum_i\langle B(e_i,e_i),H\rangle=|H|^2.
            \end{equation*}
            Furthermore, assuming $H$ is nowhere vanishing, the squared norm $|S_H|^2$ relates to the norm of the second fundamental form in the direction of $H$ by
            \begin{equation*}
                |S_H|^2=\sum_i\langle S_H(e_i),S_H(e_i)\rangle=\sum_{i,j}\langle B(e_i,e_j),H\rangle^2=|H|^2 |B_\nu|^2.
            \end{equation*}
        \end{rem}
        
    \item {\it Codazzi's equation}: 
        for $X,Y,Z\in\Gamma(TM)$ and $\xi\in\Gamma(NM)$,
        \begin{equation*}\label{codazzi_equation_general}
            \langle(\nabla_XB)(Y,Z),\xi\rangle=\langle(\nabla_YB)(X,Z),\xi\rangle,
        \end{equation*}
        where the covariant derivative of $B$ is defined by
        \begin{equation*}\label{sff_covariant_derivative_definition}
            (\nabla_XB)(Y,Z)\doteq \nabla^\perp_X\bigl(B(Y,Z)\bigr)-B(\nabla_XY,Z)-B(Y,\nabla_XZ).
        \end{equation*}
        In local coordinates, the previous equation becomes
        \begin{equation*}
            h_{ijk}^\alpha=h_{ikj}^\alpha.
        \end{equation*}
        
        \begin{rem}
            Codazzi's equation can be restated in terms of covariant derivatives acting on the shape operator. Using the definition of the covariant derivative, the definition of the shape operator, and metric compatibility, we have
            \begin{equation*}
                \langle(\nabla_XB)(Y,Z),\xi\rangle =\langle(\nabla_XS_\xi)Y,Z\rangle-\langle B(Y,Z),\nabla^\perp_X\xi\rangle.
            \end{equation*}
            Consequently, Codazzi's equation can be rewritten as
            \begin{equation}\label{codazzi_equation_shape_operator}
                \langle(\nabla_XS_\xi)Y,Z\rangle=\langle(\nabla_YS_\xi)X,Z\rangle+\langle B(Y,Z),\nabla^\perp_X\xi\rangle-\langle B(X,Z),\nabla^\perp_Y\xi\rangle.
            \end{equation}
            The last two terms in \eqref{codazzi_equation_shape_operator} vanish if $\xi$ is parallel in the normal bundle. In codimension one, the normal bundle is a trivial line bundle and such parallel sections are readily available. On the other hand, for submanifolds of higher codimension, parallel normal fields are not always available and typically obstructed by the curvature and holonomy of the normal connection.
        \end{rem}
        
    \item {\it Ricci's equation}: 
        for $X,Y\in\Gamma(TM)$ and $\eta,\xi\in\Gamma(NM)$,
        \begin{equation*}\label{ricci_equation_general}
            R^\perp(X,Y,\eta,\xi)=\langle[S_\xi,S_\eta]X,
            Y\rangle,
        \end{equation*} 
        where $R^\perp$ is the curvature tensor of the normal connection and $[S_\xi,S_\eta]$ is the commutator. In local coordinates, the previous equation reads as
        \begin{equation*}\label{ricci_equation}
            R^\perp_{ij\alpha\beta }=\sum_kh^\alpha_{ik}h^\beta_{kj}-\sum_k h^\alpha_{jk}h^\beta_{ki}.
        \end{equation*}
\end{itemize}

Combining Codazzi's equation with the commutation rule \eqref{second_covariant_derivatives_sff_commutation_rule} yields that the rough Laplacian induced by the normal connection of the second fundamental form is given by
\begin{equation}\label{sff_laplacian_1}
    \Delta^\perp h^\alpha_{ij}\doteq \sum_k h^\alpha_{ijkk}=\nabla^\perp_j\nabla^\perp_i H^\alpha+\sum_{k,m}h_{mk}^\alpha R_{mijk}+\sum_{k,m}h^\alpha_{im}R_{mkjk}+\sum_{k,\beta} h^\beta_{ik}R_{\beta\alpha jk}.
\end{equation}

\subsection{Weighted manifolds, self-shrinkers, and weighted parabolicity} 
A weighted manifold $N_h$ is a Riemannian manifold $(N^m,\langle\cdot,\cdot\rangle)$ equipped with the measure $\dvol_h\doteq e^{-h}\dvol_N$, where $h\in C^\infty(N)$. The weighted Laplacian (or $h$-Laplacian) $\Delta_h$ is defined by
\begin{equation*}
    \Delta_hu\doteq \diverg_h(\nabla u)=\Delta u-\langle\nabla h,\nabla u\rangle,
\end{equation*}
where $\diverg_hX\doteq e^h\diverg(e^{-h}X)$. The operator $\Delta_h$ is linear, elliptic, and symmetric on $L^{2}(N_h)$. The associated weighted volume (or $h$-volume) is
\begin{equation*}
    \Vol_h(N)\doteq \int_N\dvol_h=\int_N e^{-h}\dvol_N.
\end{equation*}
   A fundamental example is the Gaussian space, given by 
   \[
   G_{f}^{m}\doteq(\mathbb{R}^m, \langle\cdot,\cdot\rangle, e^{-|x|^2/4}dx),
   \]
   for which the weighted Laplacian is nothing but the Ornstein-Uhlenbeck operator
   \[
   \Delta_{f}u=\Delta u-\frac{1}{2}\langle x, \nabla u\rangle.
   \]
Let $x:M^{n}\to N_{h}^{n+p}$ be an isometric immersion. The weighted mean curvature vector is
\[
\ H_{h}\doteq H+(\nabla h)^{\bot}.
\]
The immersion is $h$-minimal if $H_{h}\equiv 0$. As a consequence of \eqref{self_shrinker_equation}, self-shrinkers of the mean curvature flow are exactly the $f$-minimal immersions into $G_{f}^{n+p}$. 

Note that such an immersion induces a weighted structure $M_f^{n}$, where $f=\frac{|x|^{2}}{4}$. We recall that by \cite{DX, CZ} the following properties are equivalent: properness of the immersion, extrinsic polynomial (indeed Euclidean) volume growth, and finiteness of the weighted volume, $\mathrm{Vol}_f(M)<+\infty$. In particular, as a consequence of a well-known result due to \cite{grigoryan_2006_weighted_manifolds}, the finiteness of the weighted volume implies that the complete weighted manifold $M_f^n$ is $f$-parabolic, i.e. any
$u \in C^0(M) \cap W^{1,2}_{\mathrm{loc}}(M)$ with
$\Delta_f u \ge 0$ and $\sup_M u < +\infty$ is constant.

\section{\texorpdfstring{Simons' equations}{Simons' equations}}\label{section_simons_equations}

\subsection{\texorpdfstring{Simons' equation for $|H|^2$}{Simons' equation for |H|2}}

Throughout the remainder of the paper, let $f$ be the function $f(x)=|x|^2/4$. Covariantly differentiating the self-shrinker equation \eqref{self_shrinker_equation} in the ambient space yields
\begin{equation}\label{ambient_H_covariant_derivative}
    \overline\nabla_{e_i}H=-\frac12 e_i+\nabla_{e_i}(\nabla f)+B(\nabla f,e_i),
\end{equation}
which implies that
\begin{equation}\label{normal_gradient_H}
    \nabla^\perp H=B(\nabla f,\cdot)
\end{equation}
and 
\begin{equation*}
    e_i|H|^2=2\langle\nabla^\perp_{e_i} H,H\rangle=2\langle S_H(\nabla f),e_i\rangle.
\end{equation*}
Consequently,
\begin{equation}\label{gradient_squared_norm_H}
    \nabla|H|^2=2S_H(\nabla f),
\end{equation}
from which it follows that
\begin{equation}\label{laplacian_squared_norm_H_1}
    \Delta|H|^2\doteq \sum_i\langle\nabla_{e_i}(\nabla|H|^2),e_i\rangle =2\sum_i\left[\langle\left(\nabla_{e_i}S_H\right)(\nabla f),e_i\rangle+\langle S_H\left(\nabla_{e_i}(\nabla f)\right),e_i\rangle\right].
\end{equation}
Using Weingarten's equation and Equation \eqref{ambient_H_covariant_derivative}, we can expand the term $\nabla_{e_i}(\nabla f)$ as
\begin{equation*}\label{gradient_nabla_f}
    \nabla_{e_i}(\nabla f) =\frac12e_i+\left(\overline\nabla_{e_i}H\right)^\top=\frac12e_i-S_H(e_i).
\end{equation*}
Substituting this back into \eqref{laplacian_squared_norm_H_1} and recalling Remark \ref{trace_norm_SH_remark}, together with \eqref{codazzi_equation_shape_operator}, we find
\begin{equation}\label{laplacian_squared_norm_H_2}
    \begin{aligned}
        \phantomsection 
        \Delta |H|^2 &= 2\sum_i\langle\left(\nabla_{e_i}S_H\right)(\nabla f),e_i\rangle+|H|^2-2|S_H|^2 
        \\&= 2\sum_i\langle\left(\nabla_{\nabla f}S_H\right)e_i,e_i\rangle+2\sum_i\left[\langle B(\nabla f,e_i),\nabla^\perp_{e_i}H\rangle-\langle B(e_i,e_i),\nabla^\perp_{\nabla f}H\rangle\right]
        \\&\qquad+|H|^2-2|S_H|^2.
    \end{aligned}
\end{equation}
A straightforward computation proves that 
\begin{equation}\label{laplacian_squared_norm_H_2_1}
    2\sum_i\langle\left(\nabla_{\nabla f}S_H\right)e_i,e_i\rangle=2\langle\nabla f,\nabla |H|^2\rangle,
\end{equation}
while identities \eqref{normal_gradient_H} and \eqref{gradient_squared_norm_H} imply that: 
\begin{align}
    2\sum_i\langle B(\nabla f,e_i),\nabla^\perp_{e_i}H\rangle&=2|\nabla^\perp H|^2,\label{laplacian_squared_norm_H_2_2}
    \\2\sum_i\langle B(e_i,e_i),\nabla^\perp_{\nabla f}H\rangle&=\langle\nabla f,\nabla |H|^2\rangle.\label{laplacian_squared_norm_H_2_3}
\end{align}

Combining \eqref{laplacian_squared_norm_H_2_1}--\eqref{laplacian_squared_norm_H_2_3}, from \eqref{laplacian_squared_norm_H_2} we derive Simons' equation for $|H|^2$:
\begin{equation}\label{f_laplacian_squared_norm_H_generic}
    \Delta_f|H|^2=2|\nabla^\perp H|^2+|H|^2-2|S_H|^2.
\end{equation}
It is worth noting that \eqref{f_laplacian_squared_norm_H_generic} remains valid regardless of whether $H$ vanishes. In particular, if $p_0\in M$ is such that $|H|(p_0)=0$ then
\begin{equation}\label{f_laplacian_squared_norm_H_generic_application}
    \Delta_f|H|^2(p_0)=2|\nabla^\perp H|^2(p_0).
\end{equation}

If we assume $H$ is nowhere vanishing, Remark \ref{trace_norm_SH_remark} allows us to rewrite equation \eqref{f_laplacian_squared_norm_H_generic} as
\begin{equation}\label{f_laplacian_squared_norm_H_nonvanishing}
    \Delta_f|H|^2=2|\nabla^\perp H|^2+|H|^2(1-2\left|B_\nu\right|^2).
\end{equation}

\subsection{\texorpdfstring{Simons' equation for $|H|$ with parallel principal normal}{Simons' equation for |H| with parallel principal normal}}

Assume $H$ is nowhere vanishing and the principal normal $\nu\doteq H/|H|$ is parallel in the normal bundle (i.e., $\nabla^\perp\nu=0$). This allows us to simplify the expression for $\Delta_f|H|$. Indeed, observe that the parallel principal normal condition is equivalent to
\begin{equation}\label{equivalence_gradient_H_and_norm_H_parallel_principal_normal}
    \nabla^\perp H=(\nabla|H|)\nu.
\end{equation}
Substituting \eqref{equivalence_gradient_H_and_norm_H_parallel_principal_normal} into \eqref{f_laplacian_squared_norm_H_nonvanishing} yields the following expression:
\begin{equation}\label{f_laplacian_norm_H_parallel_principal_normal}
    \Delta_f|H|=|H|\left(\frac12-|B_\nu|^2\right).
\end{equation}

\subsection{\texorpdfstring{Simons' equation for $|B|^2$}{Simons' equation for |B|2}}

Covariantly differentiating equation \eqref{normal_gradient_H} and using the self-shrinker equation \eqref{self_shrinker_equation} we get
\begin{equation}\label{normal_second_covariant_derivatives_H}
    \nabla^\perp_j\nabla^\perp_iH^\alpha=\frac12\nabla^\perp_j\left(\sum_k\langle x,e_k\rangle h^\alpha_{ik}\right)=\frac{1}{2}h^\alpha_{ij}-\sum_k\langle H,h_{jk}\rangle h^\alpha_{ik}+\frac{1}{2}\sum_k\langle x,e_k\rangle h^\alpha_{ijk}.
\end{equation}
Combining \eqref{sff_laplacian_1} and \eqref{normal_second_covariant_derivatives_H}, we obtain
\begin{equation}\label{intermediate_computation_1}
    \begin{aligned}
        \phantomsection
        \sum_{i,j,\alpha}h^{\alpha}_{ij}\Delta^\perp h^{\alpha}_{ij}&=\frac{|B|^2}{2}-\sum_{i,j,k,\alpha}\langle H,h_{jk}\rangle h_{ij}^\alpha h_{ik}^\alpha+\frac14\langle x,\nabla|B|^2\rangle
        \\&\qquad+\!\!\sum_{i,j,k,m,\alpha}h^{\alpha}_{ij}h^{\alpha}_{mk}R_{mijk}+\!\!\sum_{i,j,k,m,\alpha}h^{\alpha}_{ij}h^{\alpha}_{im}R_{mkjk}+\!\!\sum_{i,j,k,\alpha,\beta}h^{\alpha}_{ij}h^\beta_{ik}R_{\beta\alpha jk}.
    \end{aligned}
\end{equation}
Gauss and Ricci equations imply
\begin{equation}\label{intermediate_computation_2}
    \begin{aligned}
        \phantomsection
        \sum_{i,j,k,m,\alpha}h^{\alpha}_{ij}h^{\alpha}_{mk}R_{mijk}&=\!\!\sum_{i,j,k,m,\alpha,\beta}h^{\alpha}_{ij}h^{\alpha}_{mk}\left(h^\beta_{mj}h^\beta_{ik}-h^\beta_{mk}h^\beta_{ij}\right)
        \\&=\sum_{j,k,\alpha,\beta}(h^{\alpha}h^\beta)_{kj}(h^{\alpha}h^\beta)_{jk}-\sum_{i,k,\alpha,\beta}(h^{\alpha}h^\beta)_{ii}(h^{\alpha}h^\beta)_{kk}
        \\&=\sum_{\alpha,\beta}\tr(h^{\alpha}h^\beta)^2-\sum_{\alpha,\beta}[\tr(h^{\alpha}h^\beta)]^2.
    \end{aligned}
\end{equation}
Similarly it can be proved that:
\begin{align}
    \sum_{i,j,k,m,\alpha}h^{\alpha}_{ij}h^{\alpha}_{im}R_{mkjk}&=\sum_{i,j,m,\alpha}\langle H,h_{mj}\rangle h_{ij}^\alpha h_{im}^\alpha-\sum_{\alpha,\beta}\tr[(h^{\alpha})^2(h^\beta)^2],\label{intermediate_computation_3}
    \\ \sum_{i,j,k,\alpha,\beta}h^{\alpha}_{ij}h^\beta_{ik}R_{\beta\alpha jk}&=\sum_{\alpha,\beta}\tr(h^{\alpha}h^\beta)^2-\sum_{\alpha,\beta}\tr[(h^{\alpha})^2(h^\beta)^2].\label{intermediate_computation_4}
\end{align}

Recalling that for any matrix $A\in\mathbb R^{p,q}$ its norm is defined as $|A|^2=\tr(AA^T)=\sum_{i,j}a_{ij}^2$, and combining \eqref{intermediate_computation_1}--\eqref{intermediate_computation_4} we derive
\begin{equation}\label{f_laplacian_squared_norm_sff_generic}
    \begin{aligned}
        \phantomsection
        \Delta_f|B|^2&=2|\nabla B|^2+2\sum_{i,j,\alpha}h_{ij}^\alpha\Delta^\perp h_{ij}^\alpha-\frac12\langle x,\nabla |B|^2\rangle
        \\&=2|\nabla B|^2+|B|^2+4\sum_{\alpha,\beta}\tr(h^{\alpha}h^\beta)^2-4\sum_{\alpha,\beta}\tr[(h^{\alpha})^2(h^\beta)^2]-2\sum_{\alpha,\beta}[\tr(h^\alpha h^\beta)]^2
        \\&=2|\nabla B|^2+|B|^2-2\sum_{\alpha,\beta}|[h^\alpha,h^\beta]|^2-2\sum_{\alpha,\beta}[\tr(h^\alpha h^\beta)]^2,
    \end{aligned}
\end{equation}
where the last equality follows from the identity $|[h^\alpha,h^\beta]|^2=2\tr[(h^{\alpha})^2(h^\beta)^2]-2\tr(h^{\alpha}h^\beta)^2$. Note that \eqref{f_laplacian_squared_norm_sff_generic} is valid regardless of the vanishing of $H$. 

In what follows, we assume $H$ is nowhere vanishing and we rewrite the equation in this specific context.
We observe that
\begin{align*}
    -2\sumabt|[h^\alpha,h^\beta]|^2 
    &=-2\sumat\left(\sumbo|[h^\alpha,h^\beta]|^2-|[h^{n+1},h^\alpha]|^2\right) 
    \\&=-2\left[\sumao\sumbo|[h^\alpha,h^\beta]|^2-\sumbo|[h^{n+1},h^\beta]|^2\right]+2\sumat|[h^{n+1},h^\alpha]|^2 
    \\&=-2\sumabo|[h^\alpha,h^\beta]|^2+4\sumat|[h^{n+1},h^\alpha]|^2
    \\&=-2\sumabo|[h^\alpha,h^\beta]|^2+8\sumat\tr[(h^{n+1})^2(h^\alpha)^2]-8\sumat\tr(h^{n+1}h^\alpha)^2,
\end{align*}
therefore    
\begin{equation}\label{f_laplacian_squared_norm_sff_first_term}
    -2\sum_{\alpha,\beta}|[h^\alpha,h^\beta]|^2=-2\sumabt|[h^\alpha,h^\beta]|^2-8\sumat\tr[(h^{n+1})^2(h^\alpha)^2]+8\sumat\tr(h^{n+1}h^\alpha)^2.
\end{equation}
By an analogous computation, we derive
\begin{equation}\label{f_laplacian_squared_norm_sff_second_term}
    -2\sum_{\alpha,\beta}[\tr(h^\alpha h^\beta)]^2=-2\sumabt[\tr(h^\alpha h^\beta)]^2-4\sumat[\tr(h^{n+1}h^\alpha)]^2-2|B_\nu|^4.
\end{equation}
Substituting \eqref{f_laplacian_squared_norm_sff_first_term} and \eqref{f_laplacian_squared_norm_sff_second_term} into \eqref{f_laplacian_squared_norm_sff_generic}, the Simons' identity for $|B|^2$ becomes
\begin{equation}\label{f_laplacian_squared_norm_sff_H_nonvanishing}
    \begin{aligned}
        \phantomsection
        \Delta_f|B|^2&=2|\nabla B|^2+|B|^2-2|B_\nu|^4-2\sumabt[\tr(h^\alpha h^\beta)]^2-2\sumabt |[h^\alpha,h^\beta]|^2
        \\&\qquad +4\sumat\tr(h^{n+1}h^\alpha)^2-4\sumat\tr[(h^{n+1})^2(h^\alpha)^2]-4\sumat[\tr(h^{n+1}h^\alpha)]^2
        \\&\qquad +4\sumat \tr(h^{n+1}h^\alpha)^2-4\sumat\tr[(h^{n+1})^2(h^\alpha)^2].
    \end{aligned}
\end{equation}

\subsection{\texorpdfstring{Simons' equation for $|B_\nu|$ with parallel principal normal}{Simons' equation for |B{H/|H|}| with parallel principal normal}}

Assume that the mean curvature vector $H$ is nowhere vanishing and that the principal normal $\nu \doteq H/|H|$ is parallel in the normal bundle (i.e., $\nabla^\perp \nu = 0$). By fixing $\alpha = n+1$ and proceeding analogously to the derivation of \eqref{f_laplacian_squared_norm_sff_generic}, we establish that:
\begin{equation}\label{f_laplacian_squared_norm_principal_sff_generic}
    \Delta_f|B_\nu|^2=2|\nabla B_\nu|^2+|B_\nu|^2(1-2|B_\nu|^2)-2\sumbt |[h^{n+1},h^\beta]|^2-2\sumbt [\tr(h^{n+1}h^\beta)]^2. 
\end{equation}

The parallel principal normal condition allows for a significant simplification of \eqref{f_laplacian_squared_norm_principal_sff_generic}. Since $\nabla^\perp \nu=0$, it follows that $0=\langle R^\perp (e_i,e_j)\nu,\nu_\beta\rangle=\langle[S_\nu,S_{\nu_\beta}]e_i,e_j\rangle$
for every $i,j=1,\ldots,n$ and $\beta=n+2,\ldots,n+p$. Consequently,
\begin{equation}\label{zero_commutator}
    [h^{n+1},h^\beta]=0
\end{equation}
for every $\beta=n+2,\ldots,n+p$.

We next prove that the trace term in \eqref{f_laplacian_squared_norm_principal_sff_generic} also vanishes. Using the fact that $\nu$ is parallel in the normal bundle, along with formulas \eqref{equivalence_gradient_H_and_norm_H_parallel_principal_normal} and \eqref{f_laplacian_norm_H_parallel_principal_normal}, we derive the expression for the $f$-Laplacian of the mean curvature vector field $H$:
\begin{equation}\label{f_laplacian_H_vectori_field_1}
    \begin{aligned}
        \phantomsection
        \Delta_f^\perp H&\doteq\sum_i\left(\nabla_{e_i}^\perp\nabla_{e_i}^\perp H-\nabla^\perp_{\nabla_{e_i}e_i}H\right)-\nabla^\perp_{\nabla f}H
        \\&=(\Delta_f|H|)\nu
        \\&=H\left(\frac12-|B_\nu|^2\right)
    \end{aligned}
\end{equation}
On the other hand, as stated in Section $9$ of \cite{coldingminicozzi_Par}, the relation $LH=H$ holds generally, where $L=\Delta_f^\perp+\frac12+\sum_{k,l}\langle B(e_k,e_l),\cdot \rangle B(e_k,e_l)$. This implies  
\begin{equation}\label{f_laplacian_H_vectori_field_2}
    \begin{aligned}
        \phantomsection
        \Delta_f^\perp H&=\frac{H}{2}-\sum_{k,l}\langle B(e_k,e_l),H\rangle B(e_k,e_l)
        \\&=\frac{H}{2}-|H|\sum_{k,l}\langle B(e_k,e_l),\nu\rangle B(e_k,e_l)
        \\&=H\left(\frac12-|B_\nu|^2\right)-|H|\sum_{k,l}\sumbt\langle B(e_k,e_l),\nu\rangle \langle B(e_k,e_l),\nu_\beta\rangle\nu_\beta.
    \end{aligned}
\end{equation}
Comparing the expressions \eqref{f_laplacian_H_vectori_field_1} and \eqref{f_laplacian_H_vectori_field_2} yields
\begin{equation*}
    0=
    \sumbt\left(\sum_{k,l}h^{n+1}_{kl}h^\beta_{kl}\right)\nu_\beta=\sumbt \tr(h^{n+1}h^\beta)\nu_\beta.
\end{equation*}
Thus, we conclude that
\begin{equation}\label{zero_trace}
    \tr(h^{n+1}h^\beta)=0
\end{equation}
for all $\beta=n+2,\ldots,n+p$. 

Eventually, combining the vanishing terms from \eqref{zero_commutator} and \eqref{zero_trace} with \eqref{f_laplacian_squared_norm_principal_sff_generic}, we arrive at the simplified equation:
\begin{equation}\label{f_laplacian_squared_norm_principal_sff_simplified}
    \Delta_f|B_\nu|^2=2|\nabla B_\nu|^2+|B_\nu|^2(1-2|B_\nu|^2).
\end{equation}
Observing that $|H|^2\leq n|B_\nu|^2$ (which ensures that $|B_\nu|$ cannot vanish), we may apply the well-known Kato inequality to \eqref{f_laplacian_squared_norm_principal_sff_simplified} to deduce that
\begin{equation}\label{f_laplacian_norm_principal_sff}
    \Delta_f|B_\nu|=\frac{|\nabla B_\nu|^2-|\nabla |B_\nu||^2}{|B_\nu|}+|B_\nu|\left(\frac12-|B_\nu|^2\right)\geq |B_\nu|\left(\frac12-|B_\nu|^2\right).
\end{equation}

\subsection{\texorpdfstring{Two lower bounds for $\Delta_f|B|^2$}{Two lower bounds for Delta-f |B|2}}
As we will see in details later on, proving our main theorem requires showing that the function $|B|^2/|H|^2$ is subharmonic with respect to a suitable weighted Laplacian. To achieve this, we need suitable lower bounds for $\Delta_f|B|^2$. Following the approach in \cite{cao_xu_zhao_2024_pinching_theorem_MCF_higher_codimension}, we will establish two distinct estimates: first, a general bound for \eqref{f_laplacian_squared_norm_sff_generic} which holds unconditionally; and second, a better estimate for \eqref{f_laplacian_squared_norm_sff_H_nonvanishing} assuming the mean curvature is nowhere vanishing.

Our estimates rely on the following two algebraic results.
\begin{lem}[\cite{anmin_li_jimin_li_1992}, Theorem 1]\label{first_estimate_theorem}
    Let $A_1,\ldots,A_p$ be symmetric $(n\times n)$-matrices ($p\geq 2$). Denote with $S_{\alpha\beta}=\tr(A^T_\alpha A_\beta)$, $S_\alpha =S_{\alpha\alpha}=N(A_{\alpha})$ and $S=S_1+\ldots+S_p$, where, for a given matrix $A$, $N(A)$ denotes the square norm of $A$. Then
    \begin{equation*}
        \sum_{\alpha,\beta}N(A_\alpha A_\beta-A_\beta A_\alpha)+\sum_{\alpha,\beta}S_{\alpha\beta}^2\leq\frac32 S^2.
    \end{equation*}
\end{lem}
\begin{lem}[\cite{cheng_2002}, Lemma 3.2]\label{second_estimate_theorem}
    Let $a_1,\ldots,a_n$ and $b_{ij}$ (for $i,j=1,\ldots,n$) be real numbers satisfying $\sum_{i=1}^na_i=0$, $\sum_{i=1}^nb_{ii}=0$, $\sum_{i,j=1}^nb_{ij}^2=b$, and $b_{ij}=b_{ji}$. Then
    \begin{equation*}
        -\left(\sum_{i=1}^nb_{ii}a_i\right)^2+\sum_{i,j=1}^nb_{ij}^2a_ia_j-\sum_{i,j=1}^nb_{ij}^2a_i^2\geq -b\sum_{i=1}^na_i^2.
    \end{equation*}
\end{lem}
Keeping in mind the above two lemmas, we can then prove the following
\begin{prop} Let $x:M^n\to \mathbb R^{n+p}$ ($n\geq 2$, $p\geq 1$) be a self-shrinker. Then
\begin{equation}\label{f_laplacian_squared_norm_sff_generic_estimate}
    \Delta_f|B|^2\geq 2|\nabla B|^2+|B|^2(1-3|B|^2).
\end{equation}
Furthermore, if $|H|>0$ and $M$ satisfies a $\frac{4}{3n}$-quadratic pinching condition, then 
\begin{equation}\label{f_laplacian_squared_norm_sff_H_nonvanishing_estimate}
    \Delta_f|B|^2\geq 2|\nabla B|^2+|B|^2(1-2|B_\nu|^2)+3|B_I|^4.
\end{equation}
\end{prop}

\begin{proof} Note first that, a direct application of Lemma \ref{first_estimate_theorem} leads to
\begin{equation*}
    -2\sum_{\alpha,\beta}|[h^\alpha,h^\beta]|^2-2\sum_{\alpha,\beta}[\tr(h^\alpha h^\beta)]^2\geq -3|B|^4.
\end{equation*}
Inserting the previous estimate in \eqref{f_laplacian_squared_norm_sff_generic} immediately yields the validity of \eqref{f_laplacian_squared_norm_sff_generic_estimate}, which holds regardless of whether $H$ vanishes. 

Establishing the lower bound for \eqref{f_laplacian_squared_norm_sff_H_nonvanishing} is more involved. Again, by Lemma \ref{first_estimate_theorem} we have
\begin{equation}\label{f_laplacian_squared_norm_sff_H_nonvanishing_term_1}
    -2\sumabt[\tr(h^\alpha h^\beta)]^2-2\sumabt |[h^\alpha,h^\beta]|^2\geq -3|B_I|^4.
\end{equation}
Let us choose an orthonormal basis that diagonalizes $\mathring h^{n+1}$, i.e. $\mathring h^{n+1}=(\mathring a_i\delta_{ij})_{i,j=1,\ldots,n}$. For a fixed $\alpha\neq n+1$, a direct calculation yields:
        \begin{align*}
            \tr(\mathring h^{n+1}h^\alpha)^2&=\sum_{i,j}(\mathring h^{n+1}h^\alpha)_{ij}(\mathring h^{n+1}h^\alpha)_{ji}=\sum_{i,j}(h^\alpha_{ij})^2\mathring a_i\mathring a_j,
            \\
            \tr(\mathring h^{n+1}h^\alpha)&=\sum_{i,j}(\mathring h^{n+1})_{ij} h^\alpha_{ji}=\sum_i \mathring a_i h^\alpha_{ii},
            \\
            \tr[(\mathring h^{n+1})^2(h^\alpha)^2]&=\sum_{i,k}[(\mathring h^{n+1})^2]_{ik}[(h^\alpha)^2]_{ki}=\sum_{i,j}(h^\alpha_{ij})^2\mathring a_i^2.
        \end{align*}
        Combining these identities, we obtain
        \begin{equation}\label{f_laplacian_squared_norm_sff_H_nonvanishing_term_2}
            \begin{aligned}
                \phantomsection
                4\sumat &\tr(h^{n+1}h^\alpha)^2-4\sumat [\tr(h^{n+1}h^\alpha)]^2-4\sumat\tr[(h^{n+1})^2(h^\alpha)^2]
                \\&=4\sumat \tr(\mathring h^{n+1}h^\alpha)^2-4\sumat [\tr(\mathring h^{n+1}h^\alpha)]^2-4\sumat\tr[(\mathring h^{n+1})^2(h^\alpha)^2] 
                \\&\geq -4|\mathring B_\nu|^2|B_I|^2,
            \end{aligned}
        \end{equation}
        where the first identity follows from $\mathring h^{n+1}=h^{n+1}-\frac{|H|}{n}\Id$ and the last inequality follows from Lemma \ref{second_estimate_theorem}. In the same orthonormal basis, $h^{n+1}=( a_i\delta_{ij})_{i,j=1,\ldots,n}$, where $a_i=\mathring a_i+\frac{|H|}{n}$. Therefore,
        \begin{equation}\label{f_laplacian_squared_norm_sff_H_nonvanishing_term_3}
            \begin{aligned}
                \phantomsection
                4\sumat \tr(h^{n+1}h^\alpha)^2-4\sumat\tr[(h^{n+1})^2(h^\alpha)^2]&=4\sumat\sum_{i,j}\left[a_ia_j(h^\alpha_{ij})^2-a_i^2(h^\alpha_{ij})^2\right]
                \\&=-2\sumat\sum_{i,j}(a_i-a_j)^2(h^\alpha_{ij})^2
                \\&=-2\sumat \sum_{i,j}(\mathring a_i-\mathring a_j)^2(h^\alpha_{ij})^2
                \\&\geq -4\left(\sum_k\mathring a_k^2\right)\left(\,\, \sumat\sum_{i,j}(h_{ij}^\alpha)^2\right)
                \\&=-4|\mathring B_\nu|^2|B_I|^2.
            \end{aligned}
        \end{equation}
Combining \eqref{f_laplacian_squared_norm_sff_H_nonvanishing} and estimates \eqref{f_laplacian_squared_norm_sff_H_nonvanishing_term_1}--\eqref{f_laplacian_squared_norm_sff_H_nonvanishing_term_3} with the pinching condition $|B|^2\leq\frac{4}{3n}|H|^2$, we finally get
\begin{align*}
    \Delta_f|B|^2&\geq2|\nabla B|^2+|B|^2 -2|B_\nu|^4-3|B_I|^4-8|\mathring B_\nu|^2|B_I|^2
    \\&=2|\nabla B|^2+|B|^2-2|B_\nu|^4-3|B_I|^2(|B|^2\!-\!|B_\nu|^2)-8\left(|B_\nu|^2\!-\!\frac{|H|^2}{n}\right)|B_I|^2
    \\&\geq 2|\nabla B|^2+|B|^2-2|B_\nu|^4+|B_I|^2\left(3|B|^2-5|B_\nu|^2\right)
    \\&=2|\nabla B|^2+|B|^2(1-2|B_\nu|^2)+3|B_I|^4.
\end{align*}
\end{proof}

\section{Proof of Theorem \ref{thm:extensionsmoczyck}}

In \cite{smoczyk_2005}, K. Smoczyk established the following classification result:

\begin{thm}[\cite{smoczyk_2005}, Theorem 1.2]\label{smoczyk_2005_theorem_2}
    Let $x:M^n\to \mathbb R^{n+p}$ ($n\geq 2$, $p\geq 1$) be a complete self-shrinker with $H\neq 0$ and parallel principal normal $\nu\doteq H/|H|$. Suppose further that $M$ has uniformly bounded geometry, i.e. there exist constants $c_k$ such that $|\nabla^kB|\leq c_k$ holds uniformly on $M$ for any $k\geq 0$. Then $M$ must belong to one of the following classes:
    \begin{equation*}
        M=\Gamma\times\mathbb R^{n-1},\qquad M=\widetilde M^r\times\mathbb R^{n-r}.
    \end{equation*}
    Here, $\Gamma$ is one of the curves found by Abresch and Langer and $\widetilde M^r$ is a complete, minimal submanifold of the sphere $\mathbb{S}^{p+r-1}(\sqrt{2r})\subseteq \mathbb R^{p+r}$, where $0<r=\text{rank}(B_\nu)\leq n$ denotes the rank of the principal second fundamental form.
\end{thm}

Smoczyk's proof of Theorem \ref{smoczyk_2005_theorem_2} relies on establishing the following two facts:
\begin{align}
    &\frac{|B_\nu|^2}{|H|^2}=\text{const},\label{first_smoczyk_condition}
    \\&\left|\nabla_i|H|h^{n+1}_{jk}-|H|\nabla_ih^{n+1}_{jk}\right|=0\label{second_smoczyk_condition}.
\end{align}
These are derived via a clever integration by parts against the Gaussian kernel. This is one of the three instances in his proof where the bounded geometry condition is required, as this ensures that boundary terms vanish at infinity. Once these relations are established, the first part of the proof proceeds analogously to Theorem 1.1 in the same paper: one exploits \eqref{first_smoczyk_condition} and \eqref{second_smoczyk_condition} to demonstrate that either $\nabla^\perp H\equiv 0$ everywhere, or $B_\nu$ admits only one non-zero eigenvalue. The conclusion is then reached, in both the cases, by exploiting again the boundedness assumption on $|B|$. 

Later on, in \cite{LiWei}, Theorem \ref{smoczyk_2005_theorem_2} was improved substituting the assumption $|\nabla^{k}B|\leq c_{k}$ with the properness of the immersion and the boundedness of $|B_{I}|$. In the following we will show that indeed the latter assumption is not needed, yielding the validity of Theorem \ref{thm:extensionsmoczyck} stated in the Introduction.

Before starting with the proof of Theorem \ref{thm:extensionsmoczyck}, we state the following straightforward computational lemma, that we will need later.

\begin{lem}\label{h_laplacian_fraction_lemma}
    Given two smooth functions $u,v\in C^\infty(M)$, with $v>0$, for any $h\in C^{\infty}(M)$, it holds that
    \begin{equation*}
        \Delta_h\left(\frac uv\right)+2\langle\nabla\left(\frac uv\right),\nabla\log v\rangle=\frac{\Delta_hu}{v}-\frac{u}{v^2}\Delta_hv.
    \end{equation*}
\end{lem}

\begin{proof}[Proof of Theorem \ref{thm:extensionsmoczyck}]
Note first that the assumption that $|H|>0$, together with the fact that the immersion is proper, implies that $|B_\nu|\in L^2(M,e^{-f}\dvol)$. Indeed, as a consequence of equation \eqref{f_laplacian_norm_H_parallel_principal_normal} and of the Cauchy-Schwarz inequality, we obtain that for any $\varphi\in C^\infty_c(M)$ it holds
\begin{align*}
    \mathrm{div}_f(\varphi^2\nabla \log|H|)&=\varphi^2\Big(\frac12-|B_\nu|^2\Big)-\varphi^2|\nabla \log|H||^2+2\varphi\langle \nabla \varphi,\nabla \log|H|\rangle
    \\&\leq \varphi^2\Big(\frac12-|B_\nu|^2\Big)+|\nabla \varphi|^2.
\end{align*}
Let now $\varphi=\varphi_{R}$ be smooth cut-offs such that $\varphi_{R}=1$ on $B^{M}_{R}$, $\mathrm{supp}(\varphi_{R})\subset B^{M}_{2R}$ and $|{\nabla} \varphi|^2\leq \frac{4}{R^2}$. Then, we get
\begin{align*}
    \int_{B^{M}_{R}}|B_\nu|^2e^{-f}\dvol_M&\leq \int_{B^{M}_{2R}}\varphi_R^2|B_\nu|^2e^{-f}\dvol_M
    \\&\leq \frac12\int_{B^{M}_{2R}} \varphi_R^2e^{-f}\dvol_M+\int_{B^{M}_{2R}} |\nabla\varphi_R|^2e^{-f}\dvol_M
    \\&\leq \Big(\frac12+\frac{4}{R^2}\Big)\mathrm{Vol}_f(M).
\end{align*}
Taking the limit as $R\to+\infty$ and taking into account that $\mathrm{Vol}_f(M)<+\infty$ by the properness assumption, this yields that 
\begin{equation*}
    \int_{M}|B_\nu|^2e^{-f}\dvol_M\leq \frac12\mathrm{Vol}_f(M)<+\infty.
\end{equation*}
On the other hand, combining Lemma \ref{h_laplacian_fraction_lemma} with equations \eqref{f_laplacian_norm_H_parallel_principal_normal} and \eqref{f_laplacian_norm_principal_sff}, we have that
    \begin{align*}
        \Delta_f\left(\frac{|B_\nu|}{|H|}\right)+2\left\langle\nabla\left(\frac{|B_\nu|}{|H|}\right),\nabla \log |H|\right\rangle&=\frac{\Delta_f|B_\nu|}{|H|}-\frac{|B_\nu|}{|H|^2}\Delta_f|H|
        \\&= \frac{1}{|H||B_\nu|}\left(|\nabla B_\nu|^2-|\nabla |B_\nu||^2\right)
        \\&\geq 0.
    \end{align*}
    Setting $h=f-\log|H|^2$, this expression is equivalent to
    \begin{equation}\label{h_subharmonicity_smoczyk_extension}
        \Delta_h\left(\frac{|B_\nu|}{|H|}\right)= \frac{1}{|H||B_\nu|}\left(|\nabla B_\nu|^2-|\nabla |B_\nu||^2\right)\geq 0.
    \end{equation}
    In view of the fact that
    \begin{equation*}
        \left|\frac{B_\nu}{|H|}\right|_{L^2(M,e^{-h}\dvol)}=\left|B_\nu\right|_{L^2(M,e^{-f}\dvol)},
    \end{equation*}
    an $L^2$-Liouville theorem for $h$-subharmonic functions (cf. \cite{Pigola_vanishing_2005},\cite[Theorem 8]{Rimoldi_Class}) implies that the ratio is constant. This establishes the first condition \eqref{first_smoczyk_condition}, namely $|B_\nu|=c|H|$ for some $c>0$.
    Substituting this relationship back into \eqref{h_subharmonicity_smoczyk_extension} yields
    \begin{equation}\label{equivalence_norm_gradients_principal_sff_smoczyk}
        |\nabla B_\nu|=|\nabla |B_\nu||. 
    \end{equation}
    To establish the second condition \eqref{second_smoczyk_condition}, we begin with the following identity for a smooth $(0,2)$-tensor $T$:
    \begin{equation*}
        |T|^2 \big(|\nabla T|^2 - |\nabla |T||^2\big) = \frac{1}{2} |T_{jk} \nabla_i T_{pq} - T_{pq} \nabla_i T_{jk}|^2.
    \end{equation*}
    Thereby, if $|\nabla T|^2 = |\nabla |T||^2$ then $T_{jk} \nabla_i T_{pq} = T_{pq} \nabla_i T_{jk}$, yielding $T_{jk} \nabla_i \tr(T) = \tr(T) \nabla_i T_{jk}$. Applying this result to the scalar symmetric tensor $h^{n+1}$ and invoking \eqref{equivalence_norm_gradients_principal_sff_smoczyk} yields $\nabla_i |H| h^{n+1}_{jk} = |H| \nabla_i h^{n+1}_{jk}$, which is precisely \eqref{second_smoczyk_condition}.

    Following the proof of Theorem 1.2 in \cite{smoczyk_2005}, identities \eqref{first_smoczyk_condition} and \eqref{second_smoczyk_condition} demonstrate that either $\nabla^\perp H \equiv 0$, or $h^{n+1}$ has exactly one non-zero eigenvalue, namely $|H|$. These two cases can be analyzed separately.
    
    \begin{enumerate}[label=(\roman*)]
        \item Let us first consider the case where $\nabla^\perp H \equiv 0$. As in \cite{smoczyk_2005}, we introduce the symmetric tensor $P=\langle H,B \rangle=|H|h^{n+1}$. We point out that, because of our convention for self-shrinkers, the tensor $P$ itself is not a projection, as in \cite{smoczyk_2005}. Instead, setting 
        \[
        \ \hat{P}\doteq2P,
        \]
        one has $\hat{P}^2=\hat{P}$. Thus $\hat{P}$ is an orthogonal projection. We also set $(\hat{P}*B)_{ij}= \hat{P}_i^k B_{kj}$.
        Setting $\theta=\nabla f$, a crucial step in Smoczyk's proof relies on showing that $B_{ij}=\hat{P}_i^lB_{lj}$. Since $\hat{P}$ is a projection, this reduces to proving that $|B|^2-|\hat{P}*B|^2=0$. 
    
        If $\theta(p)=0$ at a point $p\in M$, the result follows immediately from the identity
        \begin{equation}\label{auxliary_equation_first_part_smoczyk}
            \theta^k\nabla_k(|B|^2-|\hat{P}*B|^2)=-(|B|^2-|\hat{P}*B|^2).
        \end{equation}
        If $\theta(p)\neq 0$, we consider the integral curve $\gamma$ of $\theta$ starting at $p$, defined by $\frac{d}{dt}\gamma(t)=\theta(\gamma(t))$ with $\gamma(0)=p$. Since $\nabla^{\perp}H=0$, we have $P(\theta, \cdot)=0$. Recalling that 
        \[
        \ \nabla \theta=\frac{1}{2}g-P,
        \]
        we obtain
        \[
        \ \nabla_{\theta}\theta=\frac{1}{2}\theta.
        \]
        Therefore
        \[
        \ \frac{d}{dt}|\theta|^{2}(\gamma(t))=|\theta|^{2}(\gamma(t)),
        \]
        and hence
        \[
        \ |\theta|^{2}(\gamma(t))=|\theta|^{2}(p)e^{t}>0.
        \]
        In particular, the curve is regular and has bounded speed on every bounded time interval. By completeness of $M$ and the standard continuation theorem for integral curves, $\gamma$ is well-defined for all $t\in\mathbb R$. We now define the auxiliary function $\tilde f(t) \doteq (|B|^2-|\hat{P}*B|^2)(\gamma(t))$ along $\gamma$. By applying \eqref{auxliary_equation_first_part_smoczyk}, we see that $\tilde f(t)$ satisfies the evolution equation $\frac{d}{dt}\tilde f(t)=-\tilde f(t)$, which yields:
        \begin{equation}\label{exponential_growth_auxiliary_function_smoczyk}
            (|B|^2-|\hat{P}*B|^2)(\gamma(t))=(|B|^2-|\hat{P}*B|^2)(p)e^{-t}.
        \end{equation}
        At this stage, instead of using the boundedness of $|B|$, we adapt the argument as follows. Observe that the derivative of $|x|^2$ along the curve is nonnegative:
        \begin{equation*}
            \frac{d}{dt}|x|^2(\gamma(t))=\left\langle\nabla|x|^2(\gamma(t)),\frac{d}{dt}\gamma(t)\right\rangle=4|\theta|^2(\gamma(t))\geq 0.
        \end{equation*}
        Consequently, the backward trajectory is contained in a sub-level set of $|x|^2$:
        \begin{equation*}
            \left\{\gamma(t)\:|\:t\leq 0\right\}\subseteq K\doteq\left\{\tilde p\in M\:|\:|x|^2(\tilde p)\leq |x|^2(p)\right\}.
        \end{equation*}
        By the properness of the immersion, $K$ is a compact subset of $M$. Choosing a sequence $t_n\to-\infty$, compactness guarantees that, up to passing to a subsequence, $\gamma(t_n)\to q\in K\subseteq M$. Being $|B|^2-|\hat{P}*B|^2$ a smooth function over $M$, the limit $\lim_{n\to\infty}\tilde f(t_n)=(|B|^2-|\hat{P}*B|^2)(q)$ is finite. However, if $(|B|^2-|\hat{P}*B|^2)(p) \neq 0$, equation \eqref{exponential_growth_auxiliary_function_smoczyk} implies that this same limit must be unbounded, which is a contradiction. Therefore, $(|B|^2-|\hat{P}*B|^2)(p)$ must vanish. From here, the proof concludes, exactly as in \cite{smoczyk_2005}, that $M=\tilde{M}^{r}\times \mathbb{R}^{n-r}$, where $\tilde{M}^{r}$ is a complete minimal submanifold of $\mathbb{S}^{p+r-1}(\sqrt{2r})$.
        \newline
        \item The proof of the remaining case, where $\nabla^\perp H(p_0)\neq 0$ at some $p_0\in M$, is similar. In this scenario, one first proves that $|P|^2=|H|^4$ globally, \cite{smoczyk_2005}. This implies that the tensor $P$ admits only one non-zero eigenvalue, namely $|H|^2$, and that $\nabla|H|/|\nabla|H||$ spans the corresponding eigenspace. 
        Following this, Smoczyk assumes the boundedness of $|B|$ to establish that $B_I$ vanishes identically. We adapt the argument as follows.
        
        Let $U$ be a connected component of the open set $\mathring{M} \doteq \{\tilde p \in M \:|\: \nabla |H|(\tilde p) \neq 0\}$. On $U$, we consider the two distributions 
        \begin{align*}
            \mathcal{E}_{\tilde p}U&\doteq \{X\in T_{\tilde p}U\:|\: \tilde p \in U,\: PX=|H|^2X\},
            \\\mathcal{F}_{\tilde p}U&\doteq\{X\in T_{\tilde p} U\:|\: \tilde p \in U,\: PX=0\}.
        \end{align*}
        Because $P$ is symmetric, the spectral theorem guarantees that these eigenspaces are orthogonal. We then define the vector field $\mathring\theta$ as the orthogonal projection of $\theta$ onto the distribution $\mathcal{F}U$. From
        \[
        \ P(\theta)=|H|\nabla |H|
        \]
        we obtain
        \[
        \ \mathring{\theta}=\theta-\frac{1}{|H|}\nabla |H|.
        \]
        As in \cite{smoczyk_2005}, setting
        \[
        \ B_{I}=B-\frac{1}{|H|}P\nu
        \]
        a straightforward computation gives
        \[
        \ \mathring{\theta}^{k}(B_{I})_{ki}=0.
        \]
        Differentiating this identity and using the Codazzi equation, one obtains, with our normalization,
        \[
        \ \mathring{\theta}^{k}\nabla^{\perp}_{l}(B_{I})_{ki}=-\frac{1}{2}(B_{I})_{li}.
        \]
        
        If $\mathring \theta(p)=0$, the previous equation implies $|B_I|(p)=0$. Alternatively, if $\mathring \theta(p)\neq 0$, we consider the maximal integral curve $\gamma:(a,b)\to U$ of $\mathring \theta$ with the initial condition $\gamma(0)=p$. Since $\mathcal{F}U$ is invariant under parallel transport, as in \cite{smoczyk_2005} one has
        \[
        \ \nabla_{i}\nabla_{j}|H|=\frac{\Delta|H|}{|H|^{2}}P_{ij}.
        \] 
        Since $P(\mathring{\theta})=0$, it follows that
        \[
        \ \nabla_{\mathring{\theta}}\nabla |H|=0.
        \]
        Therefore, along $\gamma$,
        \[
        \ \frac{d}{dt}|\nabla|H||^{2}(\gamma(t))=0,
        \]
        and hence
        \begin{equation}\label{Hgamma}
        \ |\nabla |H||(\gamma (t))=|\nabla |H||(p)>0.
        \end{equation}
        In particular, the integral curve cannot reach the boundary of $U$.
        
        Furthermore, in our normalization,
        \[
        \ \frac{d}{dt}|\mathring{\theta}|^{2}(\gamma(t))=|\mathring{\theta}|^{2}(\gamma(t)),
        \]
        and therefore
        \begin{equation}\label{theta0}
        \ |\mathring{\theta}|^{2}(\gamma(t))=|\mathring{\theta}|^{2}(p)e^{t}.
        \end{equation}
        Thus $\gamma$ is regular and has bounded speed on every bounded time interval. Suppose, for instance, that $b<+\infty$. By \eqref{theta0}, $\gamma$ has finite length on $[0,b)$, hence $\gamma(t)$ is Cauchy as $t\to b$. Since $M$ is complete, it converges to some $q\in M$. By \eqref{Hgamma},
        \[
        \ |\nabla |H||(q)=|\nabla|H||(p)>0,
        \]
        and therefore $q\in U$. Since $\mathring{\theta}$ is smooth in $U$, the integral curve can be extended beyond $b$, contradicting maximality. The same argument at the other endpoint shows that $a=-\infty$. Hence $(a,b)=\mathbb{R}$ and $\gamma(t)$ is contained in $U$ for all $t\in\mathbb R$. 
        
        If we assume by contradiction that $|B_I|(p)\neq 0$, reasoning as in \cite{smoczyk_2005}, one can prove that the quantity $|B_I|^2$ has an exponential growth in $t$, yielding
        \begin{equation*}
            |B_I|^2(\gamma(t))=|B_I|^2(p)e^{-t}.
        \end{equation*}
        By evaluating the derivative of $|x|^2$ along $\gamma(t)$, we obtain:
        \begin{equation*}
            \frac{d}{dt}|x|^2(\gamma(t))=\left\langle\nabla |x|^2(\gamma(t)),\frac{d}{dt}\gamma(t)\right\rangle=4\langle\theta(\gamma(t)),\mathring\theta(\gamma(t))\rangle=4|\mathring\theta|^2(\gamma(t))\geq 0.
        \end{equation*}
        The monotonicity allows us to bound the backward trajectory within a compact sub-level set of $|x|^2$. Using the properness of the immersion, we can apply exactly the same compactness and convergence arguments as in the previous scenario to reach a contradiction. Thus $B_{I}$ vanish identically and this permits to conclude the classfication as in \cite{smoczyk_2005}.
     \end{enumerate}
\end{proof}

\section{Proof of the main result}

This section is devoted to the proof of Theorem \ref{main_theorem}. 
In order to do that we will need the following two lemmas.

\begin{lem}[\cite{Huisken84}]\label{gradient_estimate_theorem_cp_5}
    For any immersed submanifold $M^n$ (without boundary) of the Euclidean space $\mathbb R^{n+p}$, the gradient estimate
    \begin{equation*}
        |\nabla B|^2\geq\frac{3}{n+2}|\nabla^\perp H|^2
        \end{equation*}
        holds.
\end{lem}

\begin{lem}\label{lemmaHequiv0}
    Let $x:M^n\to \mathbb R^{n+p}$ ($n\geq 2$, $p\geq 1$) be a self-shrinker without boundary. If $M$ is $\frac{4}{3n}$-quadratically pinched then either $H\equiv 0$ (and $M$ is a $n$-plane through the origin) or $|H|>0$.
\end{lem}

\begin{proof}
    Consider the auxiliary function
    \begin{equation*}
        \omega=\frac{4}{3n}|H|^2-|B|^2.
    \end{equation*}
    At points $p_0\in M$ where $|H|(p_0)=0$, also the auxiliary function vanishes, i.e. $\omega(p_0)=0$, because of the pinching condition. Thus equations \eqref{f_laplacian_squared_norm_H_generic_application} and \eqref{f_laplacian_squared_norm_sff_generic_estimate}, together with Lemma \ref{gradient_estimate_theorem_cp_5}, yield
    \begin{align*}
    (\Delta_f\omega)(p_0)&\leq2\left(\frac{4}{3n}|\nabla^\perp H|^2(p_0)-|\nabla B|^2(p_0)\right) -|B|^2(p_0)(1-3|B|^2(p_0))
    \\&\leq2\left(\frac{4}{3n}-\frac{3}{n+2}\right)|\nabla^\perp H|^2(p_0)
    \\&\leq \omega(p_0).
    \end{align*}
    By combining \eqref{f_laplacian_squared_norm_H_nonvanishing}, \eqref{f_laplacian_squared_norm_sff_H_nonvanishing_estimate}, and Lemma \ref{gradient_estimate_theorem_cp_5}, at points $p_1\in M$ where $|H|(p_1)>0$ we obtain the differential inequality 
    \begin{align*}
        (\Delta_f\omega)(p_1)&\leq2\left(\frac{4}{3n}|\nabla^\perp H|^2(p_1)-|\nabla B|^2(p_1)\right) +\omega(p_1)(1-2|B_\nu|^2(p_1))\nonumber
        \\&\leq 2\left(\frac{4}{3n}-\frac{3}{n+2}\right)|\nabla^\perp H|^2(p_1)+\omega(p_1)
        \\&\leq \omega(p_1).
    \end{align*} 
    Therefore, the following differential inequality is satisfied globally:
    \begin{equation}\label{differential_inequality_strong_maximum_principle_cp_5}
        \Delta_f\omega\leq 2\left(\frac{4}{3n}-\frac{3}{n+2}\right)|\nabla^\perp H|^2+\omega\leq \omega.
    \end{equation}
    The strong maximum principle for elliptic operators (see, e.g., Theorem 3.5 in \cite{gilbarg_trudinger_2001_PDEs}) applied to \eqref{differential_inequality_strong_maximum_principle_cp_5} allows us to deduce that either $\omega\equiv 0$ or $\omega>0$.

    Consider the case $\omega\equiv 0$. Then equation \eqref{differential_inequality_strong_maximum_principle_cp_5} implies that $\nabla^{\perp}H\equiv 0$ on $M$, i.e. $|H|$ is constant on $M$, being $M$ connected. This implies that either $|H|\equiv 0$ or $|H|\equiv const.>0$.

    Finally, in case $\omega>0$, the pinching condition implies that $H$ cannot vanish at any point $p_{0}$, otherwise we would have $\omega(p_0)=0$, which is a contradiction. Then $|H|>0$ everywhere.
\end{proof}

\begin{proof}[Proof of Theorem \ref{main_theorem}]
    Note first that, as a consequence of Lemma \ref{lemmaHequiv0}, either $H\equiv$ 0 (and $M$ is a $n$-plane through the origin) or $|H|>0$ everywhere. In this last case, combining Simons' estimates \eqref{f_laplacian_squared_norm_H_nonvanishing} and \eqref{f_laplacian_squared_norm_sff_H_nonvanishing_estimate} with Lemma \ref{h_laplacian_fraction_lemma} and the pinching condition $|B|^2\leq\frac{4}{3n}|H|^2$, we obtain the following estimate for the $f$-Laplacian of $|B|^2/|H|^2$:
    \begin{equation}\label{f_laplacian_ratio_BH_nonnegativity}
        \begin{aligned}
            \phantomsection
            \Delta_f\left(\frac{|B|^2}{|H|^2}\right)+2\left\langle\nabla\left(\frac{|B|^2}{|H|^2}\right),\nabla\log |H|^2\right\rangle&\geq\frac{2}{|H|^2}\left(|\nabla B|^2-\frac{|B|^2}{|H|^2}|\nabla^\perp H|^2\right)
            \\&\geq \frac{2}{|H|^2}\left(|\nabla B|^2-\frac{4}{3n}|\nabla^\perp H|^2\right)
            \\&\geq 0,
        \end{aligned}
    \end{equation}
    Here, the last inequality follows from Lemma \ref{gradient_estimate_theorem_cp_5} and the fact that $\frac{3}{n+2}\geq\frac{4}{3n}$ for every $n\geq 2,\:n\in\mathbb N$. 
  
    Let $h=f-\log|H|^4$. Observing that $\Delta_h=\Delta_f+2\langle\nabla(\cdot),\nabla\log|H|^2\rangle$, we can rewrite \eqref{f_laplacian_ratio_BH_nonnegativity} as
    \begin{equation*}
        \begin{cases}
            \Delta_h\left(\frac{|B|^2}{|H|^2}\right)\geq 0,
            \\\sup_M\frac{|B|^2}{|H|^2}\leq\frac{4}{3n}<+\infty.
        \end{cases}
    \end{equation*}
    The $h$-volume of the self-shrinker is given by
    \begin{equation*}
        \Vol_h(M)=\int_M|H|^4e^{-\frac{|x|^2}{4}}\dvol_M.
    \end{equation*}
   The relation $|H|^4\leq \frac{1}{16}|x^\perp|^4\leq \frac{1}{16}|x|^4$ implies that this integral is finite under the polynomial volume growth assumption. Hence $M$ is $h$-parabolic and this implies that the quotient $|B|^2/|H|^2$ must be constant.
    In particular 
    \[
    \ 0=\Delta_{h}\left(\frac{|B|^{2}}{|H|^{2}}\right)\geq\frac{2}{|H|^{2}}\left(|\nabla B|^{2}-\frac{|B|^{2}}{|H|^{2}}|\nabla^{\perp} H|^{2}\right)+\frac{3|B^{I}|^{4}}{|H|^{2}}.
    \]
    Using Lemma \ref{gradient_estimate_theorem_cp_5} and the pinching condition, we get
    \[
    \ |\nabla B|^{2}-\frac{|B|^{2}}{|H|^{2}}|\nabla^{\perp}H|^{2}\geq\left(\frac{3}{n+2}-\frac{4}{3n}\right)|\nabla^{\perp}H|^{2}.
    \]
    Since the coefficient on the RHS is positive for $n\geq 2$, it follows that
    \[
    \ \nabla^{\perp}H\equiv 0,\quad B_{I}\equiv0.
    \]
    Since $H=|H|\nu$ and $|H|>0$, we conclude immediately that $|H|$ is constant and $\nabla^{\perp}\nu\equiv 0$, i.e. $\nu=H/|H|$ is parallel in the normal bundle. The condition $B_I\equiv 0$ implies that the shape operators satisfy $S_{\nu_\alpha}\equiv 0$ for all $\alpha\geq n+2$. Consequently, the orthogonal complement $N_1(x)$ of the subspace $\{\xi\in N_xM \mid B_{\xi}=0\}$ in $N_xM$ is the line spanned by $\nu_{n+1}$, which is thus invariant under parallel translations in the normal bundle. By Theorem $1$ in \cite{erbacher_1971}, the self-shrinker $M$ actually lies in a $(n+1)$-dimensional affine subspace of $\mathbb R^{n+p}$. Thus Theorem \ref{colding_minicozzi_theorem} forces $M$ to be isometric either to $\Gamma\times\mathbb R^{n-1}$ or to $S^k(\sqrt{2k})\times \mathbb R^{n-k}$, with $1\leq k\leq n$. The product of Abresch-Langer curves with the Euclidean factor does not satisfy the $\frac{4}{3n}$-pinching condition, as $|B|^2=|H|^2$. In contrast, the generalized cylinders satisfy this quadratic pinching condition if and only if $\lceil\frac{3n}{4}\rceil \leq k \leq n$. To see this, let $\lambda_1, \ldots, \lambda_n$ be the principal curvatures of $S^k(\sqrt{2k})\times\mathbb{R}^{n-k}$. It is straightforward to verify that $|B|^2 = \frac{1}{2}$ and $|H|^2 = \frac{k}{2}$, thus the condition $|B|^2 \leq \frac{4}{3n}|H|^2$ implies the required bound on $k$.
\end{proof}

\section*{Data availability}
Data sharing not applicable to this article as no datasets were generated or analyzed during
the current study.

\section*{Conflict of interest statement}
On behalf of all authors, the corresponding author states that there is no conflict of interest.

\bibliographystyle{amsalpha}
\bibliography{bibliography}

\end{document}